\documentclass[12pt]{amsart}
\usepackage[margin=1in]{geometry}
\usepackage{amssymb,amsfonts,amsmath}
\usepackage{hyperref}
\usepackage{color}
\usepackage[capitalise]{cleveref}
\usepackage{listings}
\newtheorem{theorem}{Theorem}
\newtheorem{corollary}[theorem]{Corollary}
\newtheorem{lemma}[theorem]{Lemma}

\theoremstyle{remark}
\newtheorem*{remark*}{Remark}
\newtheorem{remark}[theorem]{Remark}

\numberwithin{equation}{section}
\numberwithin{theorem}{section}

\usepackage[capitalise]{cleveref}
\crefname{figure}{Figure}{Figures}
\theoremstyle{plain}
\newtheorem*{theorem*}{Theorem}
\crefname{theorem}{Theorem}{Theorems}
\crefname{lemma}{Lemma}{Lemmas}
\usepackage{constants} 
\newconstantfamily{abcon}{symbol=c}
\newconstantfamily{abcon2}{symbol=C}

\renewcommand{\epsilon}{\varepsilon}
\newcommand{\cE}{\mathcal{E}}

\renewcommand{\pmod}[1]{\, (\mathrm{mod} {\, #1})}

\newcommand{\re}{\mathrm{Re}}
\newcommand{\im}{\mathrm{Im}}

\newcommand{\R}{\mathbb{R}}

\newcommand{\asif}[1]{\textcolor{red}{\textbf{Asif:} #1}} 

\title[Refinements to the prime number theorem for arithmetic progressions]{Refinements to the prime number theorem for arithmetic progressions}
\author{Jesse Thorner}
\address{Department of Mathematics, University of Illinois, Urbana, IL 61801, United States}
\email{jesse.thorner@gmail.com}
\author{Asif Zaman}
\address{Department of Mathematics, University of Toronto, 40 St. George Street, Toronto, ON, Canada M5S 2E4}
\email{zaman@math.toronto.edu}

\begin{document}

\begin{abstract}
	We prove a version of the prime number theorem for arithmetic progressions that is uniform enough to deduce the Siegel--Walfisz theorem, Hoheisel's asymptotic for intervals of length $x^{1-\delta}$, a Brun--Titchmarsh bound, and Linnik's bound on the least prime in an arithmetic progression as corollaries.  Our proof uses the Vinogradov--Korobov zero-free region, a log-free zero density estimate, and the Deuring--Heilbronn zero repulsion phenomenon.   Improvements exist when the modulus is sufficiently powerful.
\end{abstract}

\maketitle
\vspace{-6mm}
\section{Introduction and statement of results}

\subsection{Main result} Let $q\geq 2$.  The prime number theorem for arithmetic progressions states that as $x\to\infty$, the primes $p\leq x$ equidistribute in the $\varphi(q)$ residue classes $a\pmod{q}$ with $\gcd(a,q)=1$.  Under the generalized Riemann hypothesis (GRH), primes in intervals of length $h$ equidistribute in residue classes modulo $q$ once $h/\varphi(q) \geq \sqrt{x}(\log x)^2$; that is,
{\small\begin{equation}
\label{eqn:psi}
\sum_{\substack{x-h<p\leq x\\ p\equiv a\pmod{q}}}\log p=\frac{h}{\varphi(q)}+O(\sqrt{x}(\log x)^2),\qquad x\geq q.
\end{equation}}%

Unconditional progress towards \eqref{eqn:psi} has classical origins; the latest benchmarks are described in \cref{subsec:Literature}. Each such result usually addresses a specific feature: long intervals (e.g., $h=x$), short intervals (e.g., $h=x^{1-\delta}$), small moduli (e.g., $q \leq (\log x)^{100}$), or large moduli (e.g., $q \leq x^{\delta}$).  We prove a version of the prime number theorem for arithmetic progressions that interpolates these benchmarks, giving the most uniform progress towards \eqref{eqn:psi} that current results on Dirichlet $L$-functions permit.  A major obstacle to uniformity is the possible existence of real zeros of Dirichlet $L$-functions near $s=1$.  It follows from work of McCurley \cite{McCurley} that $\prod_{\chi\pmod{q}}L(s,\chi)\neq 0$ in the region $\re(s)\geq 1-1/(13\log(q(|\im(s)|+3)))$ apart from at most one real zero, say $\beta_1$.  If $\beta_1$ exists, then it is simple, and there exists a unique nontrivial real Dirichlet character $\chi_1\pmod{q}$ such that $L(\beta_1,\chi_1)=0$.

\begin{theorem}
\label{thm:main}
Let $q\geq 2$ and $a$ be coprime integers and $4 \leq h \leq x$.  Define $\lambda$ and $\theta$ by
\[
\lambda=\lambda(x,q,a,h):=\begin{cases}
	\displaystyle \frac{1}{h}\int_{x-h}^{x}(1- \chi_1(a)t^{\beta_1-1}) dt&\mbox{if $\beta_1$ exists,}\\
	1&\mbox{otherwise,}
	\end{cases}
	\qquad \theta:=\begin{cases}
		32/37&\mbox{if $\beta_1$ exists,}\\
		7/12&\mbox{otherwise.}
	\end{cases}
\]
For all $0<\epsilon<1-\theta$, there is a constant $\Cl[abcon]{main}=\Cr{main}(\epsilon)>0$ such that if $\lambda h/\varphi(q) \geq x^{\theta+\epsilon}$, then
\[
\sum_{\substack{x-h<p\leq x \\ p\equiv a\pmod{q}}}\log p =  \frac{\lambda h}{\varphi(q)}  \Big[1+O_{\epsilon}\Big( \exp\Big( \frac{- \Cr{main} \log x}{\log q + (\log\frac{x}{h})^{\frac{2}{3}}(\log^+\log\frac{x}{h})^{\frac{1}{3}}+(\log x)^{\frac{2}{5}}(\log\log x)^{\frac{1}{5}}}\Big) \Big)\Big].
\]
The implied constant and $\Cr{main}$ are effectively computable, and $\log^+(u) = \max\{0,\log u\}$.
\end{theorem}

\begin{remark}
\label{remark1}
Let $A\geq 1$.  If $\beta_1$ exists and $q\leq (\log x)^A$, then \cref{thm:main} holds with $\theta=\frac{7}{12}$, but at a cost: $\Cr{main}$ will depend ineffectively on $A$.  We justify this in \cref{sec:flexible_pnt} using Siegel's theorem \cite[Ch. 21]{Davenport}:  For all $\epsilon>0$, there exists a constant $b_{\epsilon} >0$ such that
\begin{equation}
\label{eqn:Siegel}
1-\beta_1 \geq b_{\epsilon} q^{-\epsilon}.
\end{equation}
The constant $b_{\epsilon}$ is effectively computable when $\epsilon\geq\frac{1}{2}$ \cite[Thm 3]{Pintz}.
\end{remark}

\begin{remark}
\label{remark2}
If there exists a constant $\Cl[abcon]{relative}>0$, independent of $q$, such that $\beta_1\leq 1-\frac{\Cr{relative}}{\log q}$, then one may take $\lambda=1$ and $\theta=\frac{7}{12}$ in \cref{thm:main}.  The implied constant and $\Cr{main}$ will depend effectively on $\Cr{relative}$. We justify these assertions in \cref{sec:flexible_pnt}.
\end{remark}

If $\beta_1$ does not exist, then $\lambda=1$ and $\frac{h}{\varphi(q)}$ is the main term in \cref{thm:main}. If $\beta_1$ exists, then the notion of a ``main term'' and an ``error term'' in \cref{thm:main} depends on how one views the role of $\beta_1$. We suggest the perspective of viewing
\begin{equation}
\label{eqn:main_plus_secondary}
\frac{\lambda h}{\varphi(q)} = \frac{h}{\varphi(q)} - \frac{\chi_1(a)}{\varphi(q)}\frac{x^{\beta_1}-(x-h)^{\beta_1}}{\beta_1}
\end{equation}
as the sum of a ``primary main term'' and a ``secondary main term'', leaving
\begin{equation}
\label{eqn:error_term}
\frac{\lambda h}{\varphi(q)}\exp\Big( \frac{-\Cr{main} \log x}{\log q + (\log\frac{x}{h})^{\frac{2}{3}}(\log^+\log\frac{x}{h})^{\frac{1}{3}}+(\log x)^{\frac{2}{5}}(\log\log x)^{\frac{1}{5}}}\Big)
\end{equation}
as an ``error term''.  A similar perspective is implicit in Linnik's work on the least prime $p\equiv a\pmod{q}$ \cite{Linnik}.  In light of our piecewise definition of $\lambda$, this perspective accommodates all ranges of $q$ and $h$ under our consideration, regardless of whether $\beta_1$ exists.

If $\beta_1$ exists, then there exists $\xi\in(x-h,x)$ such that
\begin{equation}
\label{eqn:lambda_MVT}
\lambda =1-\chi_1(a)\xi^{\beta_1-1}.
\end{equation}
Furthermore, if there exists $A\geq 1$ such that $q\leq (\log x)^A$, then \eqref{eqn:Siegel} with $\epsilon=\frac{1}{3A}$ shows that
\begin{equation}
\label{eqn:lambda_asymp}
\lambda = 1+O(\exp(-b_{\frac{1}{3A}}(\log x)^{2/3})).
\end{equation}
The $O$-term in \eqref{eqn:lambda_asymp} is readily absorbed by \eqref{eqn:error_term}.  When $q$ is larger, we allow the contribution from $\beta_1$ in \eqref{eqn:main_plus_secondary} to function as a ``secondary main term'' rather than attempt to push the contribution into any sort of ``error term''.  This permits \cref{thm:main} to effectively count primes in arithmetic progressions modulo $q$ as long as the expected number of primes congruent to $a \pmod{q}$ in the interval $(x-h,x]$, namely  $\lambda h/\varphi(q)$, exceeds $x^{\theta+\epsilon}$.  If, in addition, one has $(\log x)/(\log q)\to\infty$, then the expected number of primes becomes an effective asymptotic count, even if $\beta_1$ exists.  The point is that \eqref{eqn:main_plus_secondary} and \eqref{eqn:error_term} are each scaled by $\lambda$, so the existence of $\beta_1$ proportionally affects both quantities.  This differs from traditional presentations of the prime number theorem for arithmetic progressions (e.g., \cite{Davenport}) where the contribution from $\beta_1$ is viewed as an ``error term'' instead of a ``secondary main term.''

Even if $\beta_1$ exists, we can always effectively estimate $\lambda$ using the effective upper bound $\beta_1\leq 1-b_{1/2}q^{-1/2}$ from \eqref{eqn:Siegel} as well as the bounds
\begin{equation}
\label{eqn:lambda_eff}
\tfrac{1}{8}\min\{1,(1-\beta_1)\log x\}<\lambda<2,
\end{equation}
which we prove in \cref{sec:flexible_pnt}.  
%
Therefore, even in the most restrictive case, for all $0<\epsilon<1-\theta$, there exists an effectively computable constant $\Cl[abcon]{range_const}=\Cr{range_const}(\epsilon)>0$ such that the condition $\lambda h/\varphi(q)\geq x^{\theta+\epsilon}$ in \cref{thm:main} is satisfied when $x\geq \Cr{range_const}$ and
\begin{equation}
\label{eqn:ranges_long}
h\geq x^{1-\delta_1},\qquad q\leq x^{\delta_2},\qquad \delta_1+\frac{3}{2}\delta_2\leq 1-\theta-\epsilon.
\end{equation}
If $\beta_1$ does not exist, then the $\frac{3}{2}$ can be replaced by $1$.

Taking $h=x$ and $\epsilon = \frac{1}{555}$ in \cref{thm:main}, we arrive at the following result using \eqref{eqn:ranges_long}.


\begin{corollary}
\label{cor:PNT-LongInterval}
There exists a constant $\Cl[abcon]{cor_long_2}>0$ such that if $q\geq 2$ and $a$ are coprime integers, $x\geq q^{12}$, and $\lambda=\lambda(x,q,a,x)$ is as in \cref{thm:main}, then
\[
\sum_{\substack{p\leq x \\ p\equiv a\pmod{q}}}\log p =  \dfrac{\lambda x}{\varphi(q)}\Big[1+O \Big(\exp\Big(- \Cr{cor_long_2}\frac{\log x}{\log q}\Big)+\exp\Big(-\Cr{cor_long_2} \frac{ (\log x)^{3/5}}{(\log\log x)^{1/5}}\Big)\Big)\Big].
\]
The implied constant and $\Cr{cor_long_2}$ are absolute and effectively computable.
\end{corollary}
\begin{remark}
If $\beta_1$ exists, then $\lambda(x,q,a,x) = 1-\chi_1(a)x^{\beta_1-1}/\beta_1$.
\end{remark}

\subsection{Connections with classical results}
\label{subsec:Literature}

To display the uniformity of \cref{thm:main}, we deduce the Siegel--Walfisz theorem with the best known error term, Hoheisel's asymptotic prime number theorem for short intervals with the best known error term, a Brun--Titchmarsh type bound for short intervals, and Linnik's bound on the least prime in an arithmetic progression as corollaries of \cref{thm:main}, the ineffective bound \eqref{eqn:lambda_asymp}, and the effective bound \eqref{eqn:lambda_eff}.  While all of these results can be proved individually, the point here is that \cref{thm:main} provides a single asymptotic result that is uniform enough to deduce all of them.  These deductions, which require no information on $\beta_1$ except for \eqref{eqn:Siegel} and the effective computability of $b_{1/2}$, showcase the flexibility of \cref{thm:main} in handling any combination of large, medium, or small moduli with short, medium, or long intervals.

If $A\geq 1$, and $q\leq (\log x)^A$, then \eqref{eqn:lambda_asymp} and \cref{thm:main} with $h=x$ imply that there exists an ineffective constant $c_A>0$ such that
\[
\sum_{\substack{p\leq x \\ p\equiv a\pmod{q}}}\log p = \frac{x}{\varphi(q)}\Big(1+O\Big(\exp\Big(-c_A\frac{(\log x)^{3/5}}{(\log\log x)^{1/5}}\Big)\Big)\Big).
\]
(See also Koukoulopoulos \cite[Thm 1.1]{Koukoulopoulos}.)  If $\epsilon>0$ and one adds the condition that $h\geq x^{7/12+\epsilon}(\log x)^A$, then \eqref{eqn:lambda_asymp}, \cref{thm:main}, and \cref{remark1} recover the strongest known version of Hoheisel's short interval prime number theorem \cite{Hoheisel} for arithmetic progressions to moduli $q\leq (\log x)^A$:  There exists an ineffective constant $c_{A,\epsilon}>0$ such that
\[
\sum_{\substack{x-h<p\leq x \\ p\equiv a\pmod{q}}}\log p = \frac{h}{\varphi(q)}\Big(1+O_{\epsilon}\Big(\exp\Big(-c_{A,\epsilon}\frac{(\log x)^{1/3}}{(\log\log x)^{1/3}}\Big)\Big)\Big).
\]

Consider now the situation where $q$ is large.  Applying \cref{thm:main} with $\epsilon=\frac{3}{110}$, we find that if $\lambda h/\varphi(q) \geq x^{4/5}$ and $(\log x)/(\log q)\to\infty$, then
\[
\sum_{\substack{x-h<p\leq x \\ p\equiv a\pmod{q}}}\log p\sim \frac{\lambda h}{\varphi(q)}.
\]
If $\beta_1$ exists, then $\lambda$ satisfies \eqref{eqn:lambda_eff}.  Consequently, for all fixed $\delta>0$, there exists an effectively computable constant $\Cl[abcon]{BT1}=\Cr{BT1}(\delta)>0$ such that if $x\geq q^{\Cr{BT1}}$, then
\[
\frac{h}{\varphi(q)}(1-\beta_1)\log x\ll \sum_{\substack{x-h<p\leq x \\ p\equiv a\pmod{q}}}\log p \leq (2+\delta)\frac{h}{\varphi(q)}.
\]
When $\beta_1$ does not exist, the $(1-\beta_1)\log x$ factor is deleted, and $2+\delta$ becomes $1+\delta$.  The upper bound is a weak form of the Brun--Titchmarsh theorem for short intervals (cf. \cite[Theorem 1.2]{MV_large_sieve}).  Since $\beta_1<1$, the lower bound implies that there exists a prime $p\leq q^{\Cr{BT1}}$ such that $p\equiv a\pmod{q}$, as Linnik proved \cite{Linnik}.

\subsection*{Organization}

In \cref{sec:main_thm_proof}, we state a new version of the prime number theorem for arithmetic progressions (\cref{thm:FlexiblePNT}) that depends only on a zero-free region for $\prod_{\chi\pmod{q}}L(s,\chi)$.  The proof of \cref{thm:FlexiblePNT} uses a certain log-free zero density estimate (\cref{thm:LFZDE}) and the Vinogradov--Korobov zero-free region for Dirichlet $L$-functions.  In \cref{sec:powerful}, we use \cref{thm:FlexiblePNT} and a zero-free region for $\prod_{\chi\pmod{q}}L(s,\chi)$ due to Iwaniec to refine \cref{thm:main} when $q$ is sufficiently ``powerful'' and provide an application to primes with prescribed digits.  We also contrast our results with those of Gallagher \cite{Gallagher}.   We prove \cref{thm:FlexiblePNT} in \cref{sec:flexible_pnt}.

 \subsection*{Acknowledgements}
We thank Kevin Ford, Andrew Granville, Kannan Soundararajan, and Joni Ter{\"a}v{\"a}inen  for their encouragement and helpful conversations.  We especially thank Roger Heath-Brown for his comments regarding the proof of \cref{thm:LFZDE}.

\section{Proof of \cref{thm:main}}
\label{sec:main_thm_proof}




Let $q\geq 2$, $T \geq 1$ and $0 \leq \sigma \leq 1$.  For all Dirichlet characters $\chi\pmod{q}$, we define
\[
N_{\chi}(\sigma,T):=|\{\rho=\beta+i\gamma\colon \beta,\gamma\in\R,~\beta\geq \sigma,~|\gamma|\leq T,~L(\rho,\chi)=0\}|.
\]
If $\beta_1$ exists, then we define
\[
N_{\chi_1}^*(\sigma,T):=|\{\rho=\beta+i\gamma\colon \beta,\gamma\in\R,~\beta\geq\sigma,~|\gamma|\leq T,~\rho\neq\beta_1~L(\rho,\chi)=0\}|.
\]
Let
\[
N_q(\sigma,T):=\sum_{\chi\pmod{q}}N_{\chi}(\sigma,T).
\]
If $\beta_1$ exists, then let
\[
N_q^*(\sigma,T):=\sum_{\substack{\chi\pmod{q} \\ \chi\neq\chi_1}}N_{\chi}(\sigma,T)+N_{\chi_1}^*(\sigma,T).
\]
Along with the standard theory of Dirichlet $L$-functions \cite{Davenport}, we use a log-free zero density estimate that encodes the phenomenon that if $\beta_1$ exists, then $\beta_1$ repels all other zeros of $\prod_{\chi\pmod{q}}L(s,\chi)$.
\begin{theorem}
\label{thm:LFZDE}
Let $\epsilon > 0$. Let $T\geq 1$, $q \geq 2$, and $0\leq\sigma\leq 1$.  We have the unconditional bound
\begin{equation}
\label{eqn:huxley_jutila}
N_q(\sigma,T) \ll_{\epsilon} (qT)^{(\frac{12}{5}+\epsilon)(1-\sigma)}.
\end{equation}
If $\beta_1$ exists and $\nu(U):=\min\{1,(1-\beta_1)\log U\}$, then
\begin{equation}
\label{eqn:bombieri_mod}
N_q^*(\sigma,T) \ll_{\epsilon}  \nu(qT)(qT)^{(\frac{37}{5}+\epsilon)(1-\sigma)}.
\end{equation}
\end{theorem}
\begin{remark}
Bombieri's original proof  in \cite[Ch. 6]{Bombieri2} of the existence of an absolute and effectively computable constant $\Cl[abcon]{bomb}>0$ such that $N_q^{*}(\sigma,T)\ll \nu(qT)(qT)^{\Cr{bomb}(1-\sigma)}$ is fairly complicated.  In an earlier version of this manuscript, we had pushed Bombieri's strategy to its limit, proving that one can take $\Cr{bomb}$ to be slightly smaller than $\frac{75}{4}$.   The idea behind the simple proof below, which has the added benefit of producing stronger numerical exponents, was later communicated to us by Heath-Brown.
\end{remark}
\begin{proof}
	It is well-known that \eqref{eqn:huxley_jutila} follows from the work of Huxley \cite{Huxley} and Jutila \cite{Jutila}. To deduce \eqref{eqn:bombieri_mod}, let $0<\epsilon<1$ and let $\nu_{\epsilon}$ an effectively computable and sufficiently small positive constant  depending only on $\epsilon$. If $\nu(qT) > \nu_{\epsilon}$, then the conclusion follows from \eqref{eqn:huxley_jutila} by inflating the implied constant if necessary.  Therefore, we assume that $\nu(qT) \leq \nu_{\epsilon}$, in which case our effective lower bound for $\beta_1$ implies that $qT$ is sufficiently large (depending only on $\epsilon$). The version of the Deuring--Heilbronn phenomenon proved by Jutila \cite[Thm 2]{Jutila} implies that if $\sigma>\frac{5}{6}$ and
	\begin{equation}
		\label{eqn:DH}
		\nu(qT) < \frac{1}{8}(qT)^{-(2+\epsilon)(1-\sigma)/(6\sigma-5)}(6\sigma-5),
	\end{equation} 
	then $N^*_q(\sigma,T) = 0$. Hence, \eqref{eqn:bombieri_mod} holds under these assumptions. For the remaining cases, our starting point is a trivial consequence from \eqref{eqn:huxley_jutila}, namely 
	\[
	N_q^*(\sigma,T) \leq N_q(\sigma,T) \ll_{\epsilon} \nu(qT)^{-1}  \nu(qT) (qT)^{(\frac{12}{5}+\epsilon)(1-\sigma)}.
	\]
	Since we can rescale $\epsilon$ accordingly, it therefore suffices to show $\nu(qT) \gg q^{-(5+3\epsilon)(1-\sigma)}$ to complete the proof of \eqref{eqn:bombieri_mod}. If $\sigma \leq \frac{9}{10}$, then our effective lower bound for $\beta_1$ implies that $\nu(qT) \gg q^{-1/2} \geq q^{-5(1-\sigma)}$, as required. Otherwise, if $\sigma > \frac{9}{10}$ and \eqref{eqn:DH} does not hold, then
	\[
	\nu(qT) \gg (qT)^{-(2+\epsilon) (1-\sigma) \frac{10}{4}} \gg (qT)^{-(5+3\epsilon)(1-\sigma)}, 
	\]
	as required. This proves \eqref{eqn:bombieri_mod}.
\end{proof}

Using \cref{thm:LFZDE}, we prove a highly uniform version of the prime number theorem for arithmetic progressions that depends only on a zero-free region for Dirichlet $L$-functions.

\begin{theorem} \label{thm:FlexiblePNT}
	Let $q \geq 2$ and $a$ be coprime integers, and let $4\leq h \leq x$. Let $\lambda$ and $\theta$ be as in \cref{thm:main}.  Let $0<\epsilon<1-\theta$, and let $\delta : [1,\infty) \to [0, 1]$ be a function such that for all $T\geq e$, the product $\prod_{\chi\pmod{q}}L(s,\chi)$ does not vanish in the region
	\[
	\{s\in\mathbb{C}\colon \re(s)\geq 1-\delta(T),~|\im(s)|\leq T\}
	\]
	except possibly for $\beta_1$, if it exists. If $\lambda h/\varphi(q) \geq x^{\theta+\epsilon}$, then  
	\[
	\sum_{\substack{x-h<p\leq x \\ p\equiv a\pmod{q}}}\log p = \frac{\lambda h}{\varphi(q)}\Big\{ 1 + O_{\epsilon}\Big( \sqrt{\frac{x}{h}} \frac{\log x}{\log q} \sup\Big\{ \frac{x^{-\epsilon^2 \delta(t) }}{\sqrt{t}}\colon  \frac{ex}{h} \leq t \leq \frac{x^{(1-\theta)(1-\epsilon)}}{q} \Big\} + x^{-\epsilon\theta/2} \Big) \Big\}.
	\]
	The implied constant is effectively computable and depends at most on $\epsilon$.
\end{theorem}
\begin{remark}
If $\beta_1$ exists (resp. does not exist) and the exponent $\frac{22}{5}$ (resp. $\frac{12}{5}$) in \cref{thm:LFZDE} is improved to some constant $\Cl[abcon]{improved_ZDE}>0$, then \cref{thm:FlexiblePNT} holds with $\theta = 1-\frac{1}{\Cr{improved_ZDE}}$.
\end{remark}

\begin{proof}[Proof of \cref{thm:main}]
We use the Vinogradov--Korobov zero-free region for Dirichlet $L$-functions \cite[Ch. 9]{Montgomery_ten} in \cref{thm:FlexiblePNT}.  In particular, there exists an absolute and effectively computable constant $\Cl[abcon]{ZFR-VK} > 0$ such that for $t \geq e$,
	\begin{equation}
	\label{eqn:ZFR}
	\delta(t) \geq \Cr{ZFR-VK}/(\log q + (\log t)^{2/3}(\log^+\log t)^{1/3}).
	\end{equation}
	Since $e^{-\frac{1}{u+v}} \leq e^{-\frac{1}{2\max\{u,v\}}} \leq e^{-\frac{1}{2u}} + e^{-\frac{1}{2v}}$ for $u,v > 0$, it follows that
	\[
	\dfrac{x^{-\epsilon^2 \delta(t)}}{\sqrt{t}} \leq \exp\Big(-\frac{\epsilon^2 \Cr{ZFR-VK} \log x}{2\log q} - \frac{1}{2}\log t\Big) +  \exp\Big(-\frac{\epsilon^2 \Cr{ZFR-VK} \log x}{2 (\log t)^{2/3} (\log^+\log t)^{1/3}} - \frac{1}{2}\log t\Big)
	\]
	for $t \geq e$. We maximize these two exponentials over $t \in [\frac{ex}{h}, \infty)$. The first exponential is maximized at the leftmost endpoint $t=ex/h$.  For the second exponential, notice $\frac{1}{2} \log t$ and $\frac{\epsilon^2 \Cr{ZFR-VK} \log x}{2 (\log t)^{2/3} (\log^+\log t)^{1/3}}$ are the same order of magnitude when $\log t \asymp (\log x)^{3/5} (\log\log x)^{-1/5}$. Thus, the second exponential is maximized on the interval $[\frac{ex}{h}, \infty)$ for a value of $t=t_0$ such that either $t_0=ex/h$, or $\log t_0 \asymp_{\epsilon} (\log x)^{3/5} (\log\log x)^{-1/5}$ and $t_0 \geq ex/h$. This implies that
	\begin{align*}
		\sup_{t \geq ex/h} \Big\{ \dfrac{x^{-\epsilon^2 \delta(t)}}{\sqrt{t}} \Big\} 
		& \leq \sqrt{\frac{h}{x}}  \exp\Big(-\frac{\epsilon^2 \Cr{ZFR-VK} \log x}{2\log q} \Big) + \sqrt{\frac{h}{x}}  \exp\Big(-\frac{\epsilon^2 \Cr{ZFR-VK} \log x}{2 (\log t_0)^{2/3} (\log^+\log t_0)^{1/3}} \Big) \\
		& \leq 2 \sqrt{\frac{h}{x}}  \exp\Big(-\frac{\epsilon^2 \Cr{ZFR-VK} \log x}{2\log q + 2(\log t_0)^{2/3} (\log^+\log t_0)^{1/3}} \Big), 
	\end{align*}
	since $e^{-\frac{1}{u}} + e^{-\frac{1}{v}} \leq 2e^{-\frac{1}{u+v}}$ for $u,v > 0$. 
	Therefore, by the conditions on $t_0$, there exists an effectively computable constant $\Cr{main} = \Cr{main}({\epsilon})>0$, depending at most on $\epsilon$, such that
	\begin{equation}
	\label{eqn:sup}
	\begin{aligned}
	&\sqrt{\frac{x}{h}} \frac{\log x}{\log q} \sup\Big\{ \frac{x^{-\epsilon^2 \delta(t)}}{\sqrt{t}}\colon  \frac{ex}{h} \leq t \leq \frac{x^{(1-\theta)(1-\epsilon)}}{q} \Big\}+x^{-\epsilon\theta/2}\\   
	&\leq  \frac{2\log x}{\log q}   \exp\Big(-\frac{\epsilon^2 \Cr{ZFR-VK} \log x}{2\log q + 2(\log t_0)^{2/3} (\log^+\log t_0)^{1/3}} \Big) +x^{-\epsilon\theta/2}\\
	&\ll_{\epsilon}  \exp\Big(-\frac{\Cr{main} \log x}{\log q+(\log\frac{x}{h})^{2/3}(\log^+\log\frac{x}{h})^{1/3}+(\log x)^{2/5}(\log\log x)^{1/5}}\Big). 
	\end{aligned}
	\end{equation}
	Inserting this estimate in the conclusion of \cref{thm:FlexiblePNT} proves \cref{thm:main}. 
\end{proof}

The quality of the zero-free region does not impose constraints on the range of $q$ or the length of the interval $(x-h,x]$. These ranges are determined only by the log-free zero density estimate in \cref{thm:LFZDE}. The zero-free region only influences quality of the $O$-term.

\section{Further refinements}
\label{sec:powerful}

Zero-free regions other than \eqref{eqn:ZFR} can be used with \cref{thm:FlexiblePNT}.  In some situations, this substitution can lead to improvements over \cref{thm:main}.  We consider two such examples.

\subsection{Powerful moduli}

Let $q\geq 2$, and let $d=d(q) = \prod_{p|q}p$ be the squarefree part of $q$.  It follows from work of Iwaniec \cite{Iwaniec} that in \cref{thm:FlexiblePNT}, we may take
\begin{equation}
\label{eqn:iwaniec}
\delta(t)\geq 1/(4\cdot10^4(\log d + (\log qt)^{3/4}(\log^+\log(qt))^{3/4})).
\end{equation}
Inserting \eqref{eqn:iwaniec} into \cref{thm:FlexiblePNT}, we immediately obtain another highly uniform version of the prime number theorem for arithmetic progressions.

\begin{corollary}
\label{cor:PNT-Powerful}
	Let $q \geq 2$ and $a$ be coprime integers and let $3 \leq h < x$. Let $d = \prod_{p \mid q} p$ be the squarefree part of $q$, let $\lambda$ and $\theta$ be as in \cref{thm:main}.  For all $\epsilon>0$, there exists a constant $\Cl[abcon]{powerful}=\Cr{powerful}(\epsilon) > 0$  such that if $\lambda h/\varphi(q) \geq x^{\theta+\epsilon}$ then
\[
\sum_{\substack{x-h<p\leq x \\ p\equiv a\pmod{q}}}\log p =  \frac{\lambda h}{\varphi(q)}  \Big[1+O_{\epsilon}\Big( \exp\Big( \frac{- \Cr{powerful} \log x}{\log d + (\log\frac{qx}{h})^{\frac{3}{4}}(\log^+\log\frac{qx}{h})^{\frac{3}{4}}+(\log x)^{\frac{3}{7}}(\log \log x)^{\frac{3}{7}} } \Big) \Big)\Big].
\]
The implied constant and $\Cr{powerful}$ are effectively computable.
\end{corollary}

\cref{cor:PNT-Powerful} is a substantial improvement over \cref{thm:main} when $q$ is sufficiently large relative to $d$, in which case $q$ is ``powerful.''  If $\beta_1$ exists, then by \cite[Lem 6.2]{Swaenepoel}, there exists an absolute and effectively computable constant $\Cl[abcon]{Swaenepoel}>0$ such that $\beta_1<1-\Cr{Swaenepoel}/(\sqrt{d}(\log d)^2)$.  Thus, if $q>\exp(\sqrt{d}(\log d)^2/(50\Cr{Swaenepoel}))$, then $\beta_1$ does not exist, $\lambda=1$, and if $h/\varphi(q)\geq x^{7/12+\epsilon}$, then there exists an effectively computable congmstant $\Cl[abcon]{powerful2}=\Cr{powerful2}(d)>0$ such that
\begin{equation}
\label{eqn:powerful_pnt}
\sum_{\substack{x-h<p\leq x \\ p\equiv a\pmod{q}}}\log p =  \frac{h}{\varphi(q)}  \Big[1+O_{d,\epsilon}\Big( \exp\Big( \frac{- \Cr{powerful2} \log x}{(\log\frac{qx}{h})^{\frac{3}{4}}(\log^+\log\frac{qx}{h})^{\frac{3}{4}}+(\log x)^{\frac{3}{7}}(\log \log x)^{\frac{3}{7}} } \Big) \Big)\Big].\hspace{-.077in}
\end{equation}
This can be used to study primes with prescribed digits. Namely, one can asymptotically count primes while simultaneously specifying 41.6\% of the first and last digits.

\begin{corollary}
\label{cor:digits}
	Let $A,B,N \geq 1$ and $\ell \geq 2$ be integers with $A+B< N$.  Let $\mathcal{P}_{\ell}(N)$ be the set of primes with $N$ digits in their base $\ell$ expansions.  For a prime $p\in\mathcal{P}_{\ell}(N)$, write its base $\ell$ expansion as $\sum_{j=0}^{N-1} d_j(p)\ell^{j}$.  Fix $d_0,d_1,\ldots,d_{A-1},d_{N-B},d_{N-B+1},\ldots,d_{N-1}\in\{0,1,\ldots,\ell-1\}$, with $d_{N-1}\neq 0$ and $\gcd(d_1,\ell)=1$.  For all $0<\epsilon<\frac{5}{12}$, there exists a constant $\Cl[abcon]{digits}=\Cr{digits}(\epsilon,\ell)>0$ such that if $A + B \leq ( \frac{5}{12} - \epsilon) N$, then
	\begin{multline*}
\#\{p\in\mathcal{P}_{\ell}(N)\colon \textup{$d_j(p)=d_j$ for $0\leq j\leq A-1$ and $N-B+1\leq j\leq N-1$}\}\\
=\frac{\ell^{N-A-B}}{\varphi(\ell)} \Big( 1 + O_{\epsilon,\ell}\Big( \exp\Big( - \frac{\Cr{digits} N^{1/4} }{ (\log N)^{3/4} } \Big) \Big) \Big) \Big).
	\end{multline*}
	The implied constant and $\Cr{digits}$ are effectively computable.
\end{corollary}
The proof is nearly immediate from \eqref{eqn:powerful_pnt}.  The arithmetic progression condition prescribes the $A$ least significant digits, while the short interval condition prescribes the $B$ most significant digits.  \cref{cor:digits} complements work of Swaenepoel \cite{Swaenepoel} wherein a small proportion of the digits can be prescribed without any restriction on how significant the digits are.

\subsection{Comparison with work of Gallagher}

As far as the authors know, the result closest to \cref{thm:main} in the existing literature follows from the seminal work of Gallagher \cite[Thm 7]{Gallagher}, which we now describe using our notation. 
  Recall the definition of $\lambda$ in \cref{thm:main}.  With the effective bounds in \eqref{eqn:lambda_eff}, Gallagher's work implies that there exist absolute and effectively computable constants $\Cl[abcon]{Gal1},\Cl[abcon]{Gal2}>0$ such that if
\begin{equation}
\label{eqn:Gal_range}
x/q\leq h\leq x,\qquad \exp(\sqrt{\log x})\leq q\leq x^{\Cr{Gal1}},
\end{equation}
then
\begin{equation}
\label{eqn:Gal_bound}
\Big|1-\frac{\varphi(q)}{\lambda h}\sum_{\substack{x-h<p\leq x \\ p\equiv a \pmod{q}}}\log p\Big|  \ll\begin{cases}
\displaystyle (1-\beta_1)(\log q)\exp\Big(-\Cr{Gal2}\frac{\log x}{\log q}\Big)&\mbox{if $\beta_1$ exists,}\\
\displaystyle \exp\Big(-\Cr{Gal2}\frac{\log x}{\log q}\Big)&\mbox{otherwise.} 
 \end{cases}
\end{equation}
The proof of \eqref{eqn:Gal_bound} uses a log-free bound on $N_q(\sigma,T)$ as in \cref{thm:LFZDE}, the ``standard'' zero-free region $\beta_1\neq \beta\leq 1-\frac{\Cl[abcon]{standard_ZFR}}{\log(q(|\gamma|+3))}$ and, when $\beta_1$ exists, the zero-free region \eqref{eqn:DH} derived from the zero repulsion phenomenon of Deuring and Heilbronn \cite[(27)]{Gallagher}.  This contrasts with our use of $N_q^*(\sigma,T)$ when $\beta_1$ exists and the Vinogradov--Korobov zero-free region \eqref{eqn:ZFR}.

When $\beta_1$ does not exist and $q\geq \exp(\sqrt{\log x})$, \cref{thm:main} achieves the asymptotic \eqref{eqn:Gal_bound} in a stronger range than \eqref{eqn:Gal_range}.  When $\beta_1$ exists, an exact comparison of \eqref{eqn:Gal_bound} with \cref{thm:main} depends on how large $\beta_1$ is as an explicit function of $q$.  We observe here that we can estimate \eqref{eqn:sup} using both the Vinogradov--Korobov zero-free region and \eqref{eqn:DH}.  This observation leads to a mutual refinement of \cref{thm:main} and \eqref{eqn:Gal_bound} in a range that improves upon \eqref{eqn:Gal_range}.
\begin{corollary}
\label{cor:Gal}
	Let $q\geq 2$ and $a$ be coprime integers and $4\leq h \leq x$.  If $\beta_1$ exists and $\lambda$ is as in \cref{thm:main}, then there exist $\Cl[abcon]{Gal_cor},\Cl[abcon]{Gal_cor2}\in(0,1)$ such that if $\lambda h/\varphi(q) \geq x^{\Cr{Gal_cor}}$, then 
	{\small\[
	\Big|1-\frac{\varphi(q)}{\lambda h}\sum_{\substack{x-h<p\leq x \\ p\equiv a\pmod{q}}}\log p\Big|\ll (1-\beta_1)(\log q)\exp\Big(\frac{-\Cr{Gal_cor2}\log x}{\log q + (\log\frac{x}{h})^{\frac{2}{3}}(\log^+\log\frac{x}{h})^{\frac{1}{3}}+(\log x)^{\frac{2}{5}}(\log\log x)^{\frac{1}{5}}}\Big).
	\]}%
The implied constant, $\Cr{Gal_cor}$, and $\Cr{Gal_cor2}$ are absolute and effectively computable.
\end{corollary}
\begin{proof}
Let $\theta$ be defined as in \cref{thm:main} and let $0<\epsilon<1-\theta$. Without loss of generality, we may assume that $\lambda h/\varphi(q) \geq x^{\theta+\epsilon}$ and that $q$ is sufficiently large.  The Deuring--Heilbronn zero-free region \eqref{eqn:DH} implies that there exists an absolute and effectively computable constant $\Cl[abcon]{DH2}>0$ such that $\delta(t) \geq (\Cr{DH2}\log^+(\frac{1}{(1-\beta_1)\log(qt)}))/(\log(qt)) \geq (\Cr{DH2}\log^+(\frac{1}{(1-\beta_1)\log x}))/(10\log (qt))$ for $e \leq t \leq \frac{x}{q}$. Consequently, if $e \leq t \leq \frac{x}{q}$, then
	\[
	\frac{x^{-\epsilon^2 \delta(t)}}{\sqrt{t}} \leq \sqrt{q} \exp\Big( - \frac{\epsilon^2 \Cr{DH2}\log^+\big(\frac{1}{(1-\beta_1)\log x}\big) \log x}{10\log (qt)} - \frac{1}{2}\log(qt)\Big).
	\]
The maximum value of this exponential over $t \in [\frac{ex}{h}, \infty)$ occurs at a value $t=t_1$ such that either $t_1 = ex/h$, or $\log (qt_1) \asymp_{\epsilon} ( \log^+(\frac{1}{(1-\beta_1)\log x}) \log x )^{1/2}$ and $t_1 \geq ex/h$.  Therefore, there exists an effectively computable constant $\Cl[abcon]{DH-2} = \Cr{DH-2}({\epsilon})>0$ such that if  $\lambda h/\varphi(q) \geq x^{\theta+\epsilon}$, then
	\begin{equation}
	\label{eqn:sup-DH}
	\begin{aligned}
	& \sqrt{\frac{x}{h}} \frac{\log x}{\log q} \sup\Big\{ \frac{x^{-\epsilon^2 \delta(t)}}{\sqrt{t}}\colon  \frac{ex}{h} \leq t \leq \frac{x^{(1-\theta)(1-\epsilon)}}{q} \Big\}+x^{-\epsilon\theta/2} \\
	&\ll_{\epsilon}  \frac{\log x}{\log q} \exp\Big( - \frac{\Cr{DH-2}   \log^+(\tfrac{1}{(1-\beta_1)\log x}) \log x  }{\log\frac{qx}{h} + \big( \log^+(\tfrac{1}{(1-\beta_1)\log x} ) \log x \big)^{1/2}  } \Big) + x^{-\epsilon \theta/2}.
	\end{aligned}
	\end{equation}
	
	This estimate, in conjunction with \eqref{eqn:sup}, leads to the best bound that \cref{thm:FlexiblePNT} and \eqref{eqn:DH} allow. We slightly weaken this bound for easy comparison with \eqref{eqn:Gal_bound}.  Let $0 < \eta < 1-\theta-\epsilon$. Henceforth, assume that $\lambda h/\varphi(q) \geq x^{1-\eta^2}$, in which case \eqref{eqn:lambda_eff} implies
	\[
	\log q \leq \log(qx/h) \leq (1+\epsilon) \log ( \varphi(q) x/h) + O_{\epsilon}(1) \leq (1+\epsilon) \eta^2 \log x + O_{\epsilon}(1). 
	\] 
	Using $(1-\beta_1) \log q \gg q^{-1/2}$  in \eqref{eqn:Siegel}, we obtain the bounds
	\[
	\log^+\Big(\frac{1}{(1-\beta_1)\log x} \Big)  \leq \log^+\Big(\frac{1}{(1-\beta_1)\log q} \Big)  + O(1) \ll \log q, \qquad x^{-1} \ll \big( (1-\beta_1)\log q)^{1/\eta^2}. 
	\]
	Collecting these estimates, we deduce that there exists  effectively computable constant $\Cr{DH-3} = \Cr{DH-3}(\epsilon)>0$ such that \eqref{eqn:sup-DH} is $\ll_{\eta,\epsilon} \frac{\log x}{\log q}    ( (1-\beta_1) \log x )^{\Cl[abcon]{DH-3}/\eta} =   ( \frac{\log x}{\log q}  )^3 ( (1-\beta_1) \log q)^2$.  The last equality follows upon fixing   $\eta = \frac{\Cr{DH-3}}{2}$. Taking the minimum of this estimate for \eqref{eqn:sup-DH} and the unconditional bound \eqref{eqn:sup}, the result follows since $\min\{u,v\} \leq u^{1/2} v^{1/2}$ for $u,v > 0$.  Powers of $\frac{\log x}{\log q}$ are absorbed by the exponential when $\Cr{Gal_cor2}$ is chosen to be suitably small. 
\end{proof}

\section{A flexible prime number theorem:  Proof of \cref{thm:FlexiblePNT} }
\label{sec:flexible_pnt}

We provide the details in the most complicated case where $\beta_1$ exists using $N_q^*(\sigma,T)$ and $\theta=\frac{32}{37}$.  All implied constants here are effectively computable.    We begin by proving \eqref{eqn:lambda_eff}.

\begin{lemma} \label{lem:lambda_bound}
	If $\beta_1$ exists and $4\leq x^{1/2}\leq h\leq x$, then $\frac{1}{8}\min\{1,(1-\beta_1)\log x\}<\lambda<2$.
\end{lemma}
\begin{proof}
If $\beta_1$ exists, then $\beta_1\geq 1-\frac{1}{13\log(3q)}$ per the aforementioned work of McCurley \cite{McCurley}.  If $q\leq 2\times 10^6$ and $\chi\pmod{q}$ is primitive, then $L(s,\chi)$ has no real nontrivial zeros \cite{Platt}.  Therefore, if $\beta_1$ exists, then $\frac{99}{100}<\beta_1<1$.  If $\xi\in(x-h,x)$ satisfies \eqref{eqn:lambda_MVT}, then we have the explicit equality $\xi = x(\frac{\beta_1 h/x}{1-(1-h/x)^{\beta_1}})^{1/(1-\beta_1)}$.  The bounds $\frac{x}{3}<\xi<x$ follow from our ranges for $\beta_1$ and $h$.  Since $x\geq 4$, we have $\xi>1$ and $\log x \leq 5\log\xi$, which we will use repeatedly.

To bound $\lambda$ from above, we have $\lambda = 1-\chi_1(a)\xi^{\beta_1-1} < 1+(x/3)^{\beta_1-1}<2$.  For lower bounds, we use casework.  If $\chi_1(a)=-1$, then $\min\{1,(1-\beta_1)\log x\}\leq 1< 1+\xi^{\beta_1-1}=\lambda$.  Recall that $\nu(U)=\min\{1,(1-\beta_1)\log U\}$, so $0<\nu(\xi)\leq 1$.  If $\chi_1(a)=1$ and $\nu(\xi)=1$, then
\[
	\min\{1,(1-\beta_1)\log x\}\leq 5\nu(\xi)=5<5(1-e^{-1})\leq 5(1-\xi^{\beta_1-1})=5\lambda.
\]
If $\chi_1(a)=1$ and $0<\nu(\xi)<1$, then $\lambda=1-e^{-\nu(\xi)}$, and the bound
\[
\min\{1,(1-\beta_1)\log x\}\leq 5\nu(\xi)<5\Big(1+\frac{\nu(\xi)}{2}+\frac{\nu(\xi)^2}{12}\Big)(1-e^{-\nu(\xi)})<\frac{95}{12}(1-e^{-\nu(\xi)})<8\lambda
\]
follows from the Taylor expansion of $\frac{t}{1-e^{-t}}$ at $t=0$.  These cases exhaust all possibilities.
\end{proof}

\begin{proof}[Proof of \cref{thm:FlexiblePNT}]
It suffices to consider $h\leq x-1$.  Fix $0<\epsilon<1-\theta$, and assume that
\begin{equation}
	\label{eqn:assumption_range}
	\lambda h/\varphi(q)\geq x^{\theta+\epsilon}.
\end{equation}
The explicit formula \cite[Ch. 17--20]{Davenport} implies that if $q,T,x\geq 2$, then
\begin{multline*}
	\sum_{k=1}^{\infty}\sum_{\substack{p^k\leq x \\ p\equiv a\pmod{q}}}\log p = \frac{x}{\varphi(q)}-\frac{\chi_1(a)}{\varphi(q)}\frac{x^{\beta_1}}{\beta_1}-\sum_{\chi}\frac{\bar{\chi}(a)}{\varphi(q)}\Big(\sideset{}{'}\sum_{\substack{ \beta>0 \\ |\gamma|\leq T}}\frac{x^{\rho}}{\rho}-\sum_{|\gamma|<1}\frac{1}{\rho}\Big)\\
	+O\Big( \frac{1}{\varphi(q)}\Big(\frac{x\log x+T}{T}\log q + \log x + \frac{x(\log x)(\log T)}{T}\Big)+(\log x)(\log q)+\frac{x(\log x)^2}{T}\Big),
\end{multline*}
where $\sum_{\chi}$ is a sum over characters $\chi \pmod{q}$  and $\sum'$ is a sum over non-trivial zeros $\rho=\beta+i\gamma\neq\beta_1$ of $L(s,\chi)$.  The contribution from $p^k\leq x$ with $k\geq 2$ is trivially bounded by $\ll\sqrt{x}$ using Chebyshev's bound $\sum_{p\leq x}\log p\ll x$, leading to the equation
\begin{multline}
\label{eqn:explicit_formula}
	\sum_{\substack{p\leq x \\ p\equiv a\pmod{q}}}\log p = \frac{x}{\varphi(q)}-\frac{\chi_1(a)}{\varphi(q)}\frac{x^{\beta_1}}{\beta_1}-\sum_{\chi}\frac{\bar{\chi}(a)}{\varphi(q)}\sideset{}{'}\sum_{\substack{ \beta>0 \\ |\gamma|\leq T}}\Big(\frac{x^{\rho}}{\rho}-\sum_{|\gamma|<1}\frac{1}{\rho}\Big)\\
	+O\Big( \frac{1}{\varphi(q)}\Big(\frac{x\log x+T}{T}\log q + \log x + \frac{x(\log x)(\log T)}{T}\Big)+(\log x)(\log q)+\frac{x(\log x)^2}{T}+\sqrt{x}\Big).
\end{multline}
We apply \eqref{eqn:explicit_formula} twice with $T=x^{(1-\theta)(1-\epsilon)}$ (once with $p\leq x$ and once with $p\leq x-h$).  Our assumption \eqref{eqn:assumption_range} then yields
\begin{align*}
	 \cE := \Big|\frac{\lambda h}{\varphi(q)}-\sum_{\substack{x-h<p\leq x \\ p\equiv a\pmod{q}}}\log p\Big|
		  \ll_{\epsilon} \frac{1}{\varphi(q)}\sum_{\chi}~\sideset{}{'}\sum_{\substack{\beta>0 \\ |\gamma|\leq x^{(1-\theta)(1-\epsilon)}}} \Big|\frac{x^{\rho}-(x-h)^{\rho}}{\rho} \Big| +  \frac{\lambda h}{\varphi(q)} x^{-\epsilon\theta/2}.
\end{align*}

Recall that $h\leq x-1$.  Using the bounds
\begin{align*}
\Big|\frac{x^{\rho}-(x-h)^{\rho}}{\rho}\Big|=\Big|\int_{x-h}^x \tau^{\rho-1}d\tau\Big|\ll \int_{x-h}^x \tau^{\beta-1}d\tau\ll h(x-h)^{\beta-1},\qquad \Big|\frac{x^{\rho}-(x-h)^{\rho}}{\rho}\Big|\ll \frac{x^{\beta}}{|\rho|}
\end{align*}
as well as Taylor's theorem for $\frac{x^s-(x-h)^s}{s}$ when $|s|\leq\frac{1}{\log x}$, we find that
\[
\Big|\frac{x^{\rho}-(x-h)^{\rho}}{\rho}\Big|\ll\begin{cases}
	hx^{\beta-1}\log x&\mbox{if $0<\beta<\frac{1}{4}$ and $|\gamma|\leq ex/h$,}\\
	hx^{\beta-1}&\mbox{if $\beta\geq\frac{1}{4}$ and $|\gamma|\leq ex/h$,}\\
	x^{\beta}/|\gamma|&\mbox{if $|\gamma|>ex/h$.}
\end{cases}
\]
(Separately consider the cases where $h\leq \frac{x}{2}$ and $\frac{x}{2}<h\leq x-1$.)  For $t\in\mathbb{R}$, there are $\ll\log(q(|t|+1))$ zeros of $L(s,\chi)$ such that $|t-\gamma|\leq 1$ \cite[\textsection16]{Davenport}.  Once we use this to trivially estimate the contribution from the $\rho$ with $\beta\leq \theta+\epsilon-\epsilon\theta$, it follows from \eqref{eqn:assumption_range} that
\begin{equation}
\label{eqn:pre_decomp}
	\cE
		\ll_{\epsilon} \frac{1}{\varphi(q)}\sum_{\chi}~\sideset{}{'}\sum_{\substack{\beta > \theta+\epsilon-\epsilon\theta \\ |\gamma|\leq ex/h}} \frac{x^{\beta}}{x/h}+ \frac{1}{\varphi(q)}\sum_{\chi}~\sideset{}{'}\sum_{\substack{\beta > \theta+\epsilon-\epsilon\theta \\ ex/h <|\gamma|\leq x^{(1-\theta)(1-\epsilon)} } } \frac{x^{\beta}}{|\gamma|}+ \frac{\lambda h}{\varphi(q)} x^{-\epsilon\theta/2}.
\end{equation}
We dyadically decompose $[ex/h, x^{(1-\theta)(1-\epsilon)}]$ into $O(\log x)$ subintervals, so that \eqref{eqn:pre_decomp} leads to
\begin{equation}
\label{eqn:post_decomp}
\cE	\ll_{\epsilon} \frac{1}{\varphi(q)} \sum_{k=\lceil \log(ex/h) \rceil}^{\lceil (1-\theta)(1- \epsilon) \log x \rceil}   \Big(  \sum_{\chi}~\sideset{}{'}\sum_{\substack{ \beta > \theta+\epsilon-\epsilon\theta \\ |\gamma|\leq e^k}}\frac{x^{\beta}}{e^k} \Big)+   \frac{\lambda h}{\varphi(q)} x^{-\epsilon\theta/2}.
\end{equation}

We claim that if $ex/h \leq T \leq ex^{(1-\theta)(1-\epsilon)}$, then
\begin{equation} \label{eqn:DyadicClaim}
	 \sum_{\chi}~\sideset{}{'}\sum_{\substack{\beta > \theta+\epsilon-\epsilon\theta  \\ |\gamma|\leq T}}\frac{x^{\beta}}{T} 
	 \ll_{\epsilon}    \lambda  \frac{\log x}{\log q}\Big(   \frac{ x}{\sqrt{T}}\sup\Big\{ \frac{x^{-\epsilon^2 \delta(t)}}{\sqrt{t}}\colon  \frac{ex}{h} \leq t \leq \frac{x^{(1-\theta)(1-\epsilon)}}{q} \Big\}  +  q  x^{\theta+\epsilon-\epsilon\theta} \log x  \Big).\hspace{-2.57mm}
\end{equation}
Assuming this claim, we sum over $k$, applying \eqref{eqn:post_decomp} and \eqref{eqn:DyadicClaim}, to obtain
\[
\mathcal{E} \ll_{\epsilon} \frac{\lambda h}{\varphi(q)} \Big( \sqrt{\frac{x}{h}}\frac{\log x}{\log q}  \sup\Big\{ \frac{x^{-\epsilon^2 \delta(t) }}{\sqrt{t}}\colon  \frac{ex}{h} \leq t \leq \frac{x^{(1-\theta)(1-\epsilon)}}{q} \Big\}  + \frac{q x^{\theta+\epsilon-\epsilon\theta} \log x}{h} + x^{-\epsilon\theta/2} \Big).
\]
Notice that $x^{\theta+\epsilon} \leq \frac{\lambda h}{\varphi(q)} < \frac{2 h}{\varphi(q)}$ by \cref{lem:lambda_bound} and \eqref{eqn:assumption_range}. Also, since $h \leq x$, it follows that $\log q \ll_{\epsilon} \log x$, hence $q \ll_{\epsilon} \varphi(q) x^{\epsilon\theta/4}$. Overall, this shows that $qx^{\theta+\epsilon-\epsilon\theta}\log x \ll_{\epsilon} h x^{-\epsilon\theta/2}$.  This establishes \cref{thm:FlexiblePNT} subject to our claim, which we now prove.

Let $ex/h \leq T \leq ex^{(1-\theta)(1-\epsilon)}$. By partial summation and \cref{thm:LFZDE}, we find that
\begin{equation*}
\begin{aligned}
 \sum_{\chi}~\sideset{}{'}\sum_{\substack{\beta >\theta+\epsilon-\epsilon\theta  \\ |\gamma|\leq T}}\frac{x^{\beta}}{T}  
 & \leq  \frac{x^{\theta+\epsilon-\epsilon\theta}}{T} N_q^*(\theta+\epsilon-\epsilon\theta,T)  +\frac{x \log x}{T} \int_{\theta+\epsilon-\epsilon\theta}^{1-\delta(T)}N_q^*(\sigma,T)x^{\sigma-1}\mathrm{d}\sigma \\
 & \ll_{\epsilon}  \nu(qT)\Big( \frac{x^{\theta+\epsilon-\epsilon\theta}}{T} (qT)^{(\frac{1+\epsilon}{1-\theta})(1-\theta-\epsilon+\epsilon\theta)}  +\frac{x \log x}{T} \int_{\theta+\epsilon-\epsilon\theta}^{1-\delta(T)} \Big( \frac{(qT)^{\frac{1+\epsilon}{1-\theta} }}{x} \Big)^{1-\sigma}\mathrm{d}\sigma \Big) \\
 & \ll_{\epsilon} \nu(qT)    \Big( q^{1-\epsilon^2}  x^{\theta+\epsilon-\epsilon\theta} \log x    +  \frac{x \log x}{|\log(x/(qT)^{\frac{1+\epsilon}{1-\theta}})| + 1} \cdot \frac{1}{T} \Big( \frac{(qT)^{\frac{1+\epsilon}{1-\theta}}}{x} \Big)^{\delta(T)} \Big).
\end{aligned}
\end{equation*}
Since $T \leq ex^{(1-\theta)(1-\epsilon)}$ and $\log q\ll_{\epsilon}\log x$, the definition of $\nu$ implies that $\nu(qT)\ll_{\epsilon}\nu(q)\frac{\log x}{\log q}$.  The bounds for $\lambda$ in \eqref{eqn:lambda_eff} and the inequality $q^{1-\epsilon^2}<q$ then imply that
\begin{equation}
\label{eqn:lambda_application}
 \sum_{\chi}~\sideset{}{'}\sum_{\substack{\beta >\theta+\epsilon-\epsilon\theta  \\ |\gamma|\leq T}}\frac{x^{\beta}}{T}  \ll_{\epsilon} \lambda \frac{\log x}{\log q}\Big( q x^{\theta+\epsilon-\epsilon\theta} \log x    +  \frac{x \log x}{|\log(x/(qT)^{\frac{1+\epsilon}{1-\theta}})| + 1} \cdot \frac{1}{T} \Big( \frac{(qT)^{\frac{1+\epsilon}{1-\theta}}}{x} \Big)^{\delta(T)} \Big). 
\end{equation}
\noindent
To estimate the second term in the parentheses, consider cases depending on the size of $T$.  First, suppose that $x^{(1-\theta)(1-\epsilon)}/q \leq T \leq ex^{(1-\theta)(1-\epsilon)}$.  In this range, the ratio $(qT)^{\frac{1+\epsilon}{1-\theta}}/x$ may exceed $1$.  Since $0 \leq \delta(T) \leq 1-\theta-\epsilon < (1-\theta)(1-\epsilon)$, this implies that
\begin{align*}
\frac{x \log x}{|\log(x/(qT)^{\frac{1+\epsilon}{1-\theta}})| + 1} \frac{1}{T} \Big( \frac{(qT)^{\frac{1+\epsilon}{1-\theta}}}{x} \Big)^{\delta(T)} 
& \ll 	 \frac{x \log x}{T} \min\Big[1, \Big( \frac{(qT)^{\frac{1+\epsilon}{1-\theta}}}{x} \Big)^{(1-\theta)(1-\epsilon)}  \Big] \\
& \ll_{\epsilon}  \min\Big[ \frac{x\log x}{T} , q^{1-\epsilon^2} x^{\theta+\epsilon-\epsilon\theta} T^{-\epsilon^2} \Big] 
\ll_{\epsilon}   q x^{\theta+\epsilon-\epsilon\theta} \log x, 
\end{align*}
because $T \geq x^{(1-\theta)(1-\epsilon)}/q$ and $T \geq e$ by assumption. This proves the claim \eqref{eqn:DyadicClaim} when  $x^{(1-\theta)(1-\epsilon)}/q \leq T \leq ex^{(1-\theta)(1-\epsilon)}$. Finally, if $ex/h \leq T < x^{(1-\theta)(1-\epsilon)}/q$, then 
{\small\[
\frac{(x/T) \log x}{|\log(x/(qT)^{\frac{1+\epsilon}{1-\theta}})| + 1} \Big( \frac{(qT)^{\frac{1+\epsilon}{1-\theta}}}{x} \Big)^{\delta(T)} 
\ll_{\epsilon} 	 \frac{x^{1-\epsilon^2 \delta(T)} }{T} \ll_{\epsilon} \frac{x}{\sqrt{T}}\sup\Big\{ \frac{x^{-\epsilon^2 \delta(t)}}{\sqrt{t}}\colon  \frac{ex}{h} \leq t \leq \frac{x^{(1-\theta)(1-\epsilon)}}{q} \Big\}. 
\]}%
These observations and \eqref{eqn:lambda_application} complete the proof of the claim.
\end{proof}

\cref{remark1} follows by using \eqref{eqn:lambda_asymp} instead of \eqref{eqn:lambda_eff} to deduce \eqref{eqn:lambda_application}.  \cref{remark2} follows by using the bound for $N_q(\sigma,T)$ in \cref{thm:LFZDE} instead of $N_q^*(\sigma,T)$ in our proof of \cref{thm:FlexiblePNT}; the point is that $\nu(qT)\asymp 1$ when there exists a fixed constant $\Cr{relative}>0$ such that $\beta_1\leq 1-\frac{\Cr{relative}}{\log q}$.

\bibliographystyle{abbrv}
\bibliography{Uniform_PNTAP_7.bib}

\end{document}